\documentclass[11pt]{article} 

\usepackage{amssymb,amsmath,amsthm,dsfont}
\usepackage[letterpaper,hmargin=1.0in,vmargin=1.0in]{geometry}
\usepackage{graphics} 
\usepackage[title]{appendix}
\newtheorem{theorem}{Theorem}[section]
\newtheorem{lemma}[theorem]{Lemma}
\newtheorem{proposition}[theorem]{Proposition}

\newtheorem{definition}[theorem]{Definition}

\newtheorem{conjecture}[theorem]{Conjecture}

 \newcommand{\N}{\mathcal N}
 \newcommand{\E}{\mathbb E}

\newcommand{\Hhom}{\mathcal{F}_{hom}}
\newcommand{\T}{\mathcal{T}}
\newcommand{\FF}{\mathcal{F}}

\newcommand{\ex}{\mathrm{ex}}

\begin{document} 
\title{Additive Approximation of Generalized Tur\'an Questions
\author{ Noga Alon\thanks{Department of Mathematics, Princeton University, Princeton,
		New Jersey, USA and Schools of Mathematics and Computer Science,
		Tel Aviv University, Tel Aviv, Israel. Email: {
			nalon@math.princeton.edu}.
		Research supported in part by ISF grant No. 281/17, GIF
		grant No. G-1347-304.6/2016 and the Simons Foundation. }
	 \and Clara Shikhelman \thanks{Sackler School of Mathematics,
		Tel Aviv University, Tel Aviv 69978, Israel. Email:
		clarashk@mail.tau.ac.il. Research supported in part by an ISF grant. }
}}
\setlength{\parskip}{1ex plus 0.5ex minus 0.2ex}

\maketitle 
\begin{abstract} For graphs $G$ and $T$, and a family of graphs $\FF$ let $\ex(G,T,\FF)$ denote the maximum possible number of copies of $T$ in an $\FF$-free subgraph of $G$. We investigate the algorithmic aspects of calculating and estimating this function.
We show that for every graph $T$, finite family $\FF$ and constant $\epsilon>0$  there is a polynomial time algorithm that approximates $\ex(G,T,\FF)$ for an input graph $G$ on $n$ vertices up to an additive error of $\epsilon n^{v(T)}$.
 We also consider the possibility of a better approximation, proving several positive and negative results, and suggesting a conjecture on the exact relation between $T$ and $\FF$ for which no significantly better approximation can be found in polynomial time unless $P=NP$.
%
\end{abstract}

\section{Introduction}

Many natural computational problems can be formulated as \textit{graph modification problems}. In these we are given an input graph $G$ and we aim to apply the smallest number of modifications and  get a graph which has some predefined property. 
Both the allowed modifications and the desired properties vary. The  most common modifications are 
adding, deleting, or editing edges or vertices. 
As for the desired properties, these are usually either graph properties coming from classical graph theory or properties motivated by real world applications such as Molecular
Biology \cite{CMNR}, \cite{GGKSh}, \cite{GKSh}, Circuit Design \cite{EC} or Machine Learning \cite{BBC}.




Garey and Johnson \cite{GJ} considered 18 edge and vertex modification problems. 
 Yannakakis \cite{Y81} proved that such modification problems are NP-hard for properties such as outerplanar and transitively orientable, Asano \cite{As}, and Asano and Hirata \cite{AsH} established NP-hardness for several properties expressible through forbidding families of minors or topological minors, El-Mallah and Colbourn \cite{EC} proved NP-hardness for properties defined by forbidden minors and induced subgraphs. In \cite{Y81} Yannakakis posed the question of proving NP-hardness not only for specific properties, but for general families of properties.

 In  \cite{NShSh}  Natanzon, Shamir and Sharan studied edge modification problems for hereditary properties such as being Perfect.
They showed that not only are these problems NP-hard, even finding an approximate answer, up to some constant multiplicative factor, is NP-hard. Other works have also investigated the question of approximation, see for example \cite{KR} and \cite{C}. 

In \cite{ASS} Alon, Shapira and Sudakov investigated this question for the general family of monotone graph properties\footnote{Monotone graph properties are properties that are closed under edge and vertex deletion, for example being $K_3$ free or being planar.}. The only relevant edge-modification for such properties is edge deletion. Note that any monotone property can be defined as the property of a graph being $\FF$-free where $\FF$ is an appropriate (finite or infinite) family of graphs. Thus the question becomes the following: for a graph $G$, let $\ex(G,K_2,\FF)$ denote the maximum number of edges in a subgraph of $G$ that contains no copy of $F\in\FF$. The following theorem shows that this value can be approximated as follows. 
%

%
%
%

%
\begin{theorem}[\cite{ASS}]\label{thm:ASS1}
	For any fixed $\epsilon> 0$ and any family of graphs $\mathcal{F}$ there is a deterministic algorithm that given a graph $G$ on $n$ vertices computes $\ex(G,K_2,\FF)$ up to additive error of $\epsilon n^2$ in time $O(n^2)$.

\end{theorem}

A complimentary theorem shows that excluding simple cases, a significantly better approximation is NP-hard

\begin{theorem}[\cite{ASS}]\label{thm:ASS2}
	Let $\mathcal{F}$ be a family of graphs. Then,
	
	\begin{enumerate}
		\item If there is a bipartite graph in $\mathcal{F}$, then there is a fixed $\delta > 0$ for which it is
		possible to approximate $\ex(G,K_2,\FF)$ within an additive error of $n^{2-\delta}$ in polynomial time.
		\item  On the other hand, if there are no bipartite graphs in $\mathcal{F}$, then for any fixed $\delta > 0$ it is NP-hard to
		approximate $\ex(G,K_2,\FF)$ within an additive error of $n^{2-\delta}$.
	\end{enumerate}
\end{theorem}

Note that if $G=K_n$, the complete graph on $n$ vertices, then finding $\ex(K_n,K_2,\FF)$ is the classical Tur\'an question. This question and its many variations are in the heart of extremal graph theory. Recently in \cite{ASh} the systematic study of the following variation was initiated 

%

\begin{definition}
	For graphs $G$ and $T$ and a family of graphs $\FF$ let $\ex(G,T,\FF)$ denote the maximum number of copies of $T$ in an $\FF	$-free subgraph of $G$.
\end{definition}


Following the spirit of Theorems \ref{thm:ASS1} and \ref{thm:ASS2} and the above generalization of the classical Tur\'an theorem, we consider the following question. Given a forbidden family of graphs $\FF$, a graph $T$, and a constant $\epsilon>0$, is there a polynomial time algorithm that given a graph $G$ determines $\ex(G,T,\FF)$ up to an additive error of $\epsilon n^{v(T)}$? If so, is the problem of finding a significantly better approximation NP-hard? 

We answer the first question by proving the following

\begin{theorem}\label{thm:algEx}
	For  every constant $\epsilon>0$, finite family of forbidden graphs $\FF$ and  fixed graph $T$, there is a deterministic polynomial time algorithm that given a graph $G$ on $n$ vertices approximates $\ex(G,T,\FF)$ within an additive error of $\epsilon n^{v(T)}$.
\end{theorem}
\noindent The proof is based on variants of Szemer\'edi's regularity lemma, extending the methods in \cite{ASS}.

As for the question of better approximation, the natural generalization of the easy part of Theorem \ref{thm:ASS2} holds also in our case. For a fixed graph $T$, an $m$ blow-up of $T$ is the graph obtained by replacing each vertex of $T$ by an independent set of size $m$ and each edge by a complete bipartite graph between the corresponding independent sets.

\begin{proposition}\label{prop:FsubBlowUpofT}
	Let $T$ be a fixed graph and let $\FF$ be a family of graphs such that there is a graph $F\in \FF$ which is a subgraph of a blowup of $T$. Then there is a fixed $\epsilon:=\epsilon(T,\FF)>0$ such that $\ex(G,T,\FF)$ can be calculated in polynomial time up to additive error of $n^{v(T)-\epsilon}$.
\end{proposition}

The above is a straightforward application of the following simple proposition from \cite{ASh}.
\begin{proposition}[\cite{ASh}]
	Let $T$ be a fixed graph. Then $\ex(n,T,F) =\Omega(n^{v(T)})$ if and only if $F$ is not a subgraph of
	a blow-up of $T$. Otherwise, $\ex(n, T,F) \leq n^{v(T)-\epsilon}$ for
	some $\epsilon:=\epsilon(T,F) >0$
\end{proposition} 
If indeed $F$ is a subgraph of a blow up of $T$, then for $\epsilon:=\epsilon(T.F)>0$, if $n$ is large enough we have that \[0\leq \ex(G,T,F)\leq \ex(K_n,T,F)\leq n^{v(T)-\epsilon}+0.\]
Thus $0$ is a trivial approximation of $\ex(G,T,F)$ up to additive error of $ n^{v(T)-\epsilon}$. For any family of graphs $\FF$, if $F\in \FF$ then $\ex(G,T,\FF)\leq \ex(G,T,F)$, and the required result follows.

As for the extension of the second part of Theorem \ref{thm:ASS2} we prove the following special case

\begin{theorem}\label{thm:kmBigDiff}
		Let $k,m\geq2$ be integers such that $k\geq m+2$, then for every $\epsilon>0$ approximating $\ex(G,K_m,K_k)$ up to additive error of $n^{m-\epsilon}$ is NP-hard.
	\end{theorem}
	
We believe that excluding the cases covered by Proposition \ref{prop:FsubBlowUpofT} no better approximation is possible, and so we suggest the following conjecture.

\begin{conjecture}\label{conj:mainConj}
	For every graph $T$, family of graphs $\FF$ such that no $F\in \FF$ is a subgraph of a blow-up of $T$, and $\epsilon>0$, it is NP-hard to approximate $\ex(G,T,\FF)$ up to additive error  of $n^{v(T)-\epsilon}$ for a given input graph $G$ on $n$ vertices.
\end{conjecture}

The rest of the paper is organized as follows. Section \ref{sec:Regularity} is dedicated to definitions and preliminary results, most of these concern variations of the regularity lemma.
Section \ref{sec:mainLemma} is the proof of the main lemma used to establish Theorem \ref{thm:algEx}. Sections \ref{sec:ProofOfTheorem14} and \ref{sec:ProofOfTheorem17} contain the proofs of Theorems \ref{thm:algEx} and \ref{thm:kmBigDiff} respectively, and finally Section \ref{sec:concludingRemarks} includes some further remarks concerning Conjecture \ref{conj:mainConj}  and open problems.


\section{Regularity Lemmas and Auxiliary Results}\label{sec:Regularity}

From here on let $\epsilon_1,\epsilon_2...$ and $\delta_1,\delta_2,...$ be positive constants depending on $\epsilon$, $T$ and $\FF$ and tending to zero as $\epsilon$ tends to zero. In some cases the indices refer to the lemma or theorem from which the corresponding constant arises. Additionally, let $t=v(T)$ denote the number of vertices of $T$.

\subsection{Two Versions of the Regularity Lemma}

Given a graph $G$ and two disjoint sets of vertices $A,B\subset V(G)$, let $e(A,B)$ denote the number of edges with one end point in each set and let the density of edges between the sets be defined as:
\[d(A,B)=\frac{e(A,B)}{|A|\cdot |B|}\]

\begin{definition}[$\epsilon$-Regular Pair]
	Given a graph $G$ and a pair of disjoint sets of vertices $A,B\subset V(G)$, we say that the pair $(A, B)$ is $\epsilon$-regular, if for any two subsets $A' \subseteq A$ and
	$B' \subseteq B$, such that $|A'| \geq \epsilon|A|$ and $|B'| \geq \epsilon|B|$, the inequality $|d(A', B')-d(A, B)| \leq \epsilon$ holds.
\end{definition}

\begin{definition}[$\epsilon$-Regular partition] Given a graph $G$ and $\epsilon>0$ a partition $P = \{V_i\}_{i=0}^k$ of $V(G)$ for which $|V_i|=|V_j|$ for every $i,j\geq 1$, $|V_0|<k$, and all but at most $\epsilon \binom{k}{2}$ of the pairs $(V_i, V_j )$ are $\epsilon$-regular is called an $\epsilon$-regular partition.	
\end{definition}

The following theorem is the regularity lemma of Szemer\'edi \cite{Sz}. It is convenient to use the following version which appears, for example, in \cite{FR}, Theorem 3.7.

\begin{theorem}[\cite{Sz},\cite{FR}]\label{thm:Ref}
	
	For any $\epsilon>0$ and integer $k$ there exist integers $N_{\ref{thm:Ref}}:=N(\epsilon,k)$ and $K_{\ref{thm:Ref}}:=K(\epsilon,k)$ such that the following holds. If $|V|>N_{\ref{thm:Ref}}$ then for any partition $V=V_0\cup...\cup V_k$ with $|V_0|<k$ and $|V_1|=...=|V_k|$ and any graph $G$ on the vertex set $V$ there exists a partition $V=U_0\cup ... \cup U_{k'}$ such that
	\begin{enumerate}
		\item $|U_0| < k' <K_{\ref{thm:Ref}}$.
		\item For each $U_i$ with $i\geq 1$ there is $V_j$ such that $U_i\subset V_j$.
		\item $|U_1|=...=|U_{k'}|$.
		\item $U_0\cup ... \cup U_{k'}$ is an $\epsilon$-regular partition of $G$.
	\end{enumerate}
\end{theorem}

We will also need the following algorithmic version of the regularity lemma that uses a stronger notion of regularity.
For the following see \cite{ASSt} and \cite{RS}

\begin{definition}[$f$-Regular partition]
	For a function $f : \mathbb{N}\to(0, 1)$ and a graph $G$ we say that a partition $P = \{V_i\}_{i=0}^k$ of $V(G)$ such that $|V_0|<k$ and $|V_i|=|V_j|$ for every $i,j\geq 1$ is an $f$-regular partition if \textbf{all} pairs $(V_i, V_j ), 1 ≤\leq i < j \leq  k$, are $f(k)$-regular.
\end{definition}

\begin{theorem}[\cite{ASSt}] \label{thm:fRegPart}
	For every $l, \epsilon> 0$ and non-increasing function $f : \mathbb{N}\to(0, 1)$, there is
	an integer $K_{\ref{thm:fRegPart}}:= K_{\ref{thm:fRegPart}}(f, \epsilon, l)$ so that for  a given  graph $G$ on $n$ vertices, with $n>K_{\ref{thm:fRegPart}}$, one can add  or remove at most $\epsilon n^2$
	edges of $G$ and get a graph $G_0$
	that has an $f$-regular partition of order $k$, where
	$l \leq  k \leq K_{\ref{thm:fRegPart}}$. Furthermore, the needed changes and the partition can be found in polynomial time.
\end{theorem} 

Note that in \cite{ASSt} the definition of $f$-regular partition is slightly different. There is no set $V_0$ and the sizes of the sets are such that for every $i,j$, $\big||V_i|-|V_j|\big|\leq 1$. 
It is easy to check that the two versions are equivalent.
For simplicity we will use the version which has a set $V_0$ and $k$ sets of equal size throughout the paper.  

\subsection{Definitions of Weighted Graphs and Partition Graphs}

We use the following definitions for weighted graphs.
\begin{definition}
$\mbox{ }$
\begin{enumerate}
	\item A weighted graph $W$ is a graph on $n$ vertices with a weight function $w:E(K_n)\to [0,1]$, where we identify between $w(e)=0$ and $e\notin E(W)$. 
	
	\item For a fixed graph $T$ let $\mathcal{T}$ be the the set of copies of $T$ in $W$ and let
	$$\N(W,T)=\sum_{T\in \mathcal{T}\mbox{     }} \prod_{(v_i,v_j)\in E(T)}w(v_i,v_j).$$ Call $W$ a $T$-free graph if $\N(W,T)=0$.
	
	\item Let $W'$ be a weighted graph on $V(W)$ with a weight function $w':E(W')\to [0,1]$. We say that $W'$  is a conventional subgraph of $W$ if $\forall e\in K_n$ either $w'(e)= w(e)$ or $w'(e)=0$. 

\end{enumerate}
\end{definition}

Given a graph $G$ and an $\epsilon$-regular partition of it, say $P=\{V_i\}$, we would like to associate a weighted graph to $G$ and $P$ and to relate subgraphs of the weighted graph to subgraphs of $G$. To do this we use the following definitions:
\begin{definition}
	$\mbox{ }$
\begin{enumerate}
	\item Given a graph $G$ and an $\epsilon$-regular partition $P=\{V_i\}_{i=0}^k$ of its vertices define the $(\epsilon,d)$-partition graph $W$ to be a weighted graph on $k$ vertices $\{v_1,...,v_k\}$. The weight function is $w (v_i,v_j)=d(V_i,V_j)$ if $(V_i,V_j)$ is an $\epsilon$-regular pair with density at least $d$ and  $w(v_i,v_j)=0$ otherwise.
	
	\item For a conventional subgraph of $W$, say $W'$, let $G_{W'}$ be the following subgraph of $G$ on the same set of vertices $V(G)=V(G_{W'})$. An edge $e=\{u,u'\}\in E(G)$ is also an edge of $G_{W'}$ if and only if $u\in V_i, u'\in V_j$ and $w'(v_i,v_j)>0$. The vertices of $V_0$ form an independent set.
\end{enumerate}
\end{definition}
%

\subsection{Embedding and Counting Copies of Fixed Graphs}

For a graph $R$ and an integer $h$ let $R(h)$ be the $h$-blowup of $R$, that is the graph obtained by replacing each vertex with an independent set of size $h$, and each edge with a complete bipartite graph between the corresponding independent sets.


\begin{theorem}[Embedding Lemma, see, e.g., \cite{KSSS}]\label{lem:embedding}
Given $d>\epsilon> 0$, a graph $R$, and a
positive integer $m_{\ref{lem:embedding}}$, construct a graph $G$ by replacing every vertex of $R$ by
$m_{\ref{lem:embedding}}$ vertices, and by replacing each edges of $R$ by an $\epsilon$-regular pairs of density at least
$d$. Let $F$ be a subgraph of $R(h)$ and maximum degree $\Delta > 0$, 
let $\delta=d-\epsilon$ and $\epsilon_0=\delta^{\Delta}/(2+\Delta)$. If $\epsilon\leq \epsilon_0$ and $h-1\leq \epsilon_0 m_{\ref{lem:embedding}}$ then $F\subset G$.
\end{theorem}

As we aim to use the embedding lemma with an $(\epsilon,d)$ partition graph, we need to choose $d$ to suite a family of graphs $\FF$. Let $\Delta(\FF)=\max_{F\in \FF}\{\Delta(F)\}$, and let $v(\FF)=\max_{F\in \FF}\{v(F)\}$. For a fixed $\epsilon$ and family $\FF$ let $\epsilon_0=\epsilon$ and let $d_{\ref{lem:embedding}}:=d(\epsilon,\Delta(\FF))=\delta+\epsilon$ as they appear in Theorem \ref{lem:embedding}. Note that $d_{\ref{lem:embedding}}$ tends to zero as $\epsilon$ tends to zero. Let $m_{\ref{lem:embedding}}:=m_{\ref{lem:embedding}}(\epsilon,\FF)=v(\FF)\frac{1}{\epsilon}$. 

\begin{lemma}\label{lem:TCopiesUsingSamePartTwice}
	Let $W$ be a weighted graph on $n$ vertices, $\epsilon>0$ be a constant and $P=\cup_{i=1}^k V_i$ a partition of $V(W)$ into $k\geq \frac{1}{\epsilon}$ parts of equal size. Let $\T'$ be the set of all copies of $T$ using at least two vertices from the same $V_i$. Then $$\sum_{T\in \T'} \prod_{(v_i,v_j)\in E(T)}w(v_i,v_j)<\epsilon n^t.$$
\end{lemma}
\begin{proof}
	As $w(v_i,v_j)\leq 1$ note that $\sum_{T\in \T'} \prod_{(v_i,v_j)\in E(T)}w(v_i,v_j)\leq|\T'|$, and so it is enough to bound the number of copies of $T$ in $\T'$. A copy from $\T'$ uses vertices from at most $t-1$ sets of the partition, and in each such set there are $\frac{n}{k}$ vertices. Thus
	\[
	|\T'|<k^{t-1}(\frac{n}{k})^{t}\leq n^t \frac{1}{k}
	\]
	as $\epsilon\geq \frac{1}{k}$ the required result follows.
\end{proof}

Note that the above lemma can be used for unweighted graphs by choosing  $w(v_i,v_j)\in \{0,1\}$ appropriately.


\begin{lemma}\label{lem:countingLemma}
	Let $T$ be a graph on $t$ vertices, let $G$ be a graph on $n$ vertices, $P=V_0\cup\{V_i\}_{i=1}^k$ an $\epsilon$-regular partition of its vertices and let $W$ be its $(\epsilon,d_{\ref{lem:embedding}})$ partition graph. Then there exists $\delta_{\ref{lem:countingLemma}}:=\delta_{\ref{lem:countingLemma}}(\epsilon,t)$, that tends to zero when $\epsilon$ tends to zero, such that
	\begin{equation*}
	|\N(G,T)-|V_i|^{t}\N(W,T)|<\delta_{\ref{lem:countingLemma}} n^{t}
	\end{equation*}

\end{lemma}

\begin{proof}

Let $\N^P(G,T)$ be the number of copies of $T$ in $G$ using at most one vertex from each set $\{V_i\}_{i=1}^k$ and not using vertices from $V_0$, edges between sets $(V_i,V_j)$  which are not $\epsilon$-regular or with density smaller than $d_{\ref{lem:embedding}}$. Call these copies \textit{conventional}. 
 
There are at most $\epsilon\binom{ k}{2} (\frac{n}{k})^2 n^{t-2} \leq \epsilon n^{t}$ copies of $T$ using an edge between irregular pair, there are at most $\binom{k}{2}d_{\ref{lem:embedding}}(\frac{n}{k})^2n^{t-2}\leq d_{\ref{lem:embedding}} n^t$ copies using edges between sets of small density and there are at most $kn^{t-1}$ copies using a vertex from $V_0$. Together with Lemma \ref{lem:TCopiesUsingSamePartTwice} we get that for $\delta_0>0$ that tends to zero when $\epsilon$ tends to zero 
\begin{equation}\label{eq:usingSamePartitionDoesntMatter}
|\N(G,T)-\N^P(G,T)|<\delta_0 n^t.
\end{equation}

Let $\T=\{T_i\}$ be the set of all copies of $T$ in $W$. For a given copy $T_i$ let $V(T_i)=\{v_{i_1},..,v_{i_t}\}$ be the sets of vertices in $W$ that $T_i$ uses and let $V_{i_1},..,V_{i_t}$ be the corresponding sets in $P$ and let $\N(G,T_i)$ be the number of copies of $T$ in $G$ using a vertex from $V_{i_j}$ for the role of $v_{i_j}$.

A simplified version of Lemma 1.6 from \cite{DLR} is the following
\begin{lemma}\label{lem:countingTs}
	Let $T$ be a fixed graph on $t$ vertices, and let $G$ be a graph with an $\epsilon$-regular partition into $t$ parts $P=V_1,...,V_t$. Let $d_{i,j}$ be the density between the sets $V_i$ and $V_j$, and assume that $d_{i,j}=0$ for every $\{i,j\}\not \in E(T)$. Then for $\delta_{\ref{lem:countingTs}}:=\delta_{\ref{lem:countingTs}}(\epsilon,t)$ that tends to zero as $\epsilon$ rends to zero,
	\begin{equation*}
	|\N(G,T)-|V_i|^t\prod_{(i,j)\in E(T)}d_{i,j}|\leq \delta_{\ref{lem:countingTs}} |V_i|^t
	\end{equation*}
\end{lemma}

In our case, when counting the number of copies of $T_i$ we focus only on the sets $V_{i_1},..,V_{i_t}$, and assume that the density between two sets that do not correspond to an edge of $T$ is zero. Thus, 
\[
|\N(G,T_i)-|V_i|^t\prod_{(v_{i_k},v_{i_j})\in E(T_i)}w(v_{i_k},v_{i_j})|\leq \delta_{\ref{lem:countingTs}} |V_i|^t
\]

%

Each conventional copy of $T$ in $G$ can be mapped to a copy $T_i$ in $W$ by mapping each vertex $v$ of $T$ to the set $V_i$ it is contained in. Thus $\N^P(G,T)=\sum_{T_i\in \T}\N(G,T_i)$. 
%
%

Finally, note that $|V_i|\leq \frac{n}{k}$, and together with the above we get
\begin{align*}
|\N^P(G,T)-|V_i|^t\N(W,T)|&=|\sum_{T_i\in \T}\big( \N(G,T_i)-|V_i|^t\prod_{(v_i,v_j)\in E(T_i)}w(v_i,v_j)\big)| \\
&\leq \sum_{T_i\in \T} |\N(G,T_i)-|V_i|^t\prod_{(v_i,v_j)\in E(T_i)}w(v_i,v_j)|\\
&\leq |\T|\delta_{\ref{lem:countingTs}} |V_i|^t\leq k^t\delta_{\ref{lem:countingTs}} (\frac{n}{k})^t=\delta_{\ref{lem:countingTs}}n^t
\end{align*}

This together with (\ref{eq:usingSamePartitionDoesntMatter}) gives the required result. 

\end{proof}

\subsection{Homomorphisms of Graphs and Homomorphism Freeness}

\begin{definition}
	Given graphs $G$ and $F$ we say that the function $\varphi:V(F)\to V(G)$ is a homomorphism from $F$ to $G$ if it maps edges to edges.
\end{definition}

\begin{definition}
		Given a graph $G$ and a family of graphs $\FF$, we say that $G$ is $\FF$-homomorphism-free  ($\Hhom$-free, for short), if there is no homomorphism from $F$ to $G$ for any $F\in \FF$.
\end{definition}

\begin{definition}
For a weighted graph $W$, a graph $T$ and a family of graphs $\FF$ define the homomorphism extremal number to be $$\ex_{hom}(W,T,\FF)=\max \{\N(W_0,T):\mbox{ }W_0\mbox{ is  a conventional subgraph of }W\mbox{ and is }\Hhom\mbox{-free}  \}$$
\end{definition}

%
%


\begin{lemma}\label{lem:F-Hom-free}
	Let $G$ be a graph, $\FF$ a finite family of graphs, $\epsilon>0$,  $P=\cup_{i=0}^k V_i$ an $\epsilon$-regular partition of $G$ and $W$ its $(\epsilon,d_{\ref{lem:embedding}})$-partition graph. If $v(G)\geq m_{\ref{lem:embedding}}k$, then the following holds:
	\begin{enumerate}
		\item If $G$ is $\FF$-free then $W$ is $\Hhom$-free.
		\item If $W$ is $\Hhom$-free then $G_W$ is $\FF$-free.
	\end{enumerate}
	
\end{lemma}

\begin{proof}
	For the first part, assume towards contradiction that there is a graph $F\in \FF$ such that there is a homomorphism of it into $W$. To apply Lemma \ref{lem:embedding} first note that this means that $F$ is a subgraph of $W(v(\FF))$, which is the graph obtained by replacing every vertex of $W$ by an independent set of size $v(\FF)$ and every edge with positive weight by a complete bipartite graph between the corresponding independent sets. Furthermore, as $n\geq m_{\ref{lem:embedding}}k$ and the edges correspond to $\epsilon$-regular pairs in the partition with density at least $d_{\ref{lem:embedding}}$ we get that indeed $G$ must contain a copy of $F$. This is a contradiction as we assumed that $G$ is $\FF$-free.
	
	As for the second part, assume that there is a copy of some $F\in \FF$ in $G_W$. Let $V(F)=\{v_1,...,v_f\}$ be the vertices of this copy and assume that in the copy in $G_W$ vertex $v_j$ comes from the set $V_{k_j}$. If $(v_i,v_j)$ is an edge in the copy $F$, this means that there is an edge between the vertices $v_{k_i},v_{k_j}$ in $W$. Thus the mapping $v_i\to v_{k_i}$ is a homomorphism. As we assumed that $W$ is $\Hhom$-free this is a contradiction.
\end{proof}

\subsection{Auxiliary Lemmas}

Here we prove simple lemmas to be used in the proof of Theorem \ref{thm:algEx}.

\begin{lemma}\label{lem:removingSmallNumOfCopiesOfT}
	Let $G$ and $G'$ be graphs on $n$ vertices, such that one can edit (i.e. add or remove) $\delta n^2$ edges of $G$ and get $G'$. Then for every fixed graph $T$ and family of graphs $\FF$,
	$$|ex(G,T,\FF)-ex(G',T,\FF)|\leq \delta n^{t}.$$
\end{lemma}

\begin{proof}
	Let $E$ be the set of edges we need to delete from $G$ to make it a subgraph of $G'$, and let $G_0$ be the subgraph of $G$ which is $\FF$-free and has the maximum possible number of copies of $T$. If $G'_0$ is the subgraph of $G_0$ obtained by deleting from it any edges from $E$ then $G'_0$ is an $\FF$-free subgraph of $G'$ and the following holds 
	\[\ex(G',T,\FF)\geq \N(G'_0,T)\geq\N(G_0,T)-\delta n^2n^{t-2}= \ex(G,T,\FF)-\delta n^{t}.\] As this is symmetric for $G$ and $G'$ we get the needed result.
\end{proof}

\begin{lemma}\label{lem:minusEpsilonAndMult}
	For every $\epsilon>0$ and $r\in \mathbb{N}$ there is $\delta_{\ref{lem:minusEpsilonAndMult}}:=\delta_{\ref{lem:minusEpsilonAndMult}}(\epsilon,r)$ that tends to zero as $\epsilon$ tends to zero such that for every set of constants $0<\alpha_1,...,\alpha_r<1$
	\[
	\prod_{i=1}^r(\alpha_i-\epsilon)>\prod_{i=1}^r\alpha_i-\delta_{\ref{lem:minusEpsilonAndMult}}
	\]
\end{lemma}

\begin{proof}
	As $\alpha_i<1$ the following holds
	\begin{align*}
	\prod_{i=1}^r(\alpha_i-\epsilon)=&\prod_{i=1}^r\alpha_i+\sum_{i=1}^r(-\epsilon)^i\sum_{I\subset[r],|I|=r-i}\prod_{j\in I}\alpha_j\\
	\geq&\prod_{i=1}^r\alpha_i-\sum_{i=1}^r\epsilon^i\sum_{I\subset[r],|I|=r-i}\prod_{j\in I}\alpha_j\\
	\geq&\prod_{i=1}^r\alpha_i-\sum_{i=1}^r\epsilon^i\binom{r}{i}\\
	=&\prod_{i=1}^r\alpha_i-\big((1+\epsilon)^r-1\big)
	\end{align*}
	Thus taking $\delta_{\ref{lem:minusEpsilonAndMult}}=\big((1+\epsilon)^r-1\big)$ gives the needed result.
\end{proof}

\section{The Main Lemma}\label{sec:mainLemma}

The algorithm for Theorem \ref{thm:algEx} gets a graph $G$ and uses the regularity lemma to get a partition graph $W$, and then solves the extremal question for $W$. To prove the correctness of the algorithm we need to show that indeed the answer for $W$ gives a good estimate for the original extremal question on the graph $G$.

\begin{lemma}\label{lem:mainLemma}
	Let $T$ be a fixed graph, $\FF$ a finite family of graphs, $\epsilon>0$ and $k \in \mathbb{N}$. Then there exists $\delta_{\ref{lem:mainLemma}}:=\delta_{\ref{lem:mainLemma}}(\epsilon,T,\FF)$ that tends to zero as $\epsilon$ tends to zero, for which the following holds.
	Let $G$ be a graph on $n$ vertices, let $P$ be an $f(k)$-partition of $G$ into $k$ parts as in Theorem \ref{thm:fRegPart} where $f(k)=\min\{\epsilon,k\cdot K_{\ref{thm:Ref}}^{-1}\}$ and let $W$ be the $(f(k),d_{\ref{lem:embedding}})$ partition graph it gives. Then 
	$$\ex(G,T,\FF)\leq (\frac{n}{k})^{t}\ex_{hom}(W,T,\FF)+\delta_{\ref{lem:mainLemma}} n^{t}$$
\end{lemma}

\begin{proof}[Proof of Lemma \ref{lem:mainLemma}]
	Let $W_0$ be a conventional subgraph of $W$ such that $\N(W_0,T)=\ex_{hom}(W,T,\FF)$. Assume towards contradiction that $G$ has an $\FF$-free subgraph $G_{\ex}$ such that 
	
	\begin{equation}\label{eq:contradictionAssumption}
	\N(G_{\ex},T)> \N(W_0,T)(\frac{n}{k})^{t}+\delta_{\ref{lem:mainLemma}} n^{t}.
	\end{equation}
	where $\delta_{\ref{lem:mainLemma}}$ is chosen in the end and tends to zero as $\epsilon$ tends to zero.

	Using Theorem \ref{thm:Ref} 
	we can find a refinement of $P=\{V_i\}_{i=1}^k$ which is an $\epsilon$-regular partition of $G_{\ex}$ into $k_0\cdot k< K_{\ref{thm:Ref}}$ parts where $k_0\geq \frac{1}{\epsilon}$. 
	Call this partition $Q=V_0\cup\{\{V_{i,j}\}_{j=1}^{k_0}\}_{i=1}^{k}$, where $V_{i,j}\subset V_i$. Note that as $|V_0|<K_{\ref{thm:Ref}}$ and all of the other sets are of equal size, we get that 
	$$|V_{i,j}|> K^{-1}_{\ref{thm:Ref}}\cdot n=  k(K_{\ref{thm:Ref}})^{-1}\cdot \frac{n}{k} \geq kK^{-1}_{\ref{thm:Ref}}\cdot |V_i|,$$  
	and so as $P$ is an $f$-regular partition with $f(k)\leq  k K_{\ref{thm:Ref}}^{-1}$ it holds that
	\begin{equation}\label{eq:isARegSet}
	d_G(V_i,V_j)>d_G(V_{i,l_i},V_{j,l_j})-\epsilon
	\end{equation}
	for every choice of $1<l_i,l_j<k_0$ and $1<i,j<k$, where $d_G$ is the density in $G$. 
	
	Let $W^*$ be the $(\epsilon,d_{\ref{lem:embedding}})$ partition graph of $Q$ and $G_{\ex}$. Note that by Lemma \ref{lem:countingLemma} and by (\ref{eq:contradictionAssumption})  this means that 
	\[
	\left(\frac{n}{k\cdot k_0}\right)^t \N(W^*,T)\geq \N(G_{\ex},T)-\delta_{\ref{lem:countingLemma}} n^{t} \geq\left(\frac{n}{k}\right)^t\N(W_0,T)+(\delta_{\ref{lem:mainLemma}}-\delta_{\ref{lem:countingLemma}}) n^{t}
	\]
	and in particular
	\begin{equation}\label{eq:diffOfWeightedGraphs}
	\left(\frac{1}{k_0}\right)^t\N(W^*,T)\geq \N(W_0,T)+(\delta_{\ref{lem:mainLemma}}-\delta_{\ref{lem:countingLemma}})k^t.
	\end{equation}
	
	To obtain a contradiction to the existence of $G_{\ex}$ we first prove the following lemma
	
	\begin{lemma}\label{lem:choosingTheSubgraphForCont}
		There exists a choice of $r_i\in[k_0]$
		for every $i\in [k]$ such that the following holds. Let $W^{**}$ be the subgraph of $W^*$ spanned by the vertices $v_{1, r_1},...,v_{k,r_k}$ then
		$$
		k_0^t \N(W^{**},T)\geq \N(W^*,T)-\epsilon_{\ref{lem:TCopiesUsingSamePartTwice}} (k\cdot k_0)^t 
		$$
	\end{lemma}
	\begin{proof}
		Let $\mathcal{T}^*$ be the set of copies of $T$ in $W^*$ such that for every two vertices of the copy, $v_{i,k_i}$ and $v_{j,k_j}$, $i\ne j$.
		Using Lemma \ref{lem:TCopiesUsingSamePartTwice} we get that 
		\begin{align*}
		\sum_{T\in \T^*}\prod_{(v_i,v_j)\in E(T)} w^{*}(v_i,v_j)>\N(W^*,T)-\epsilon_{\ref{lem:TCopiesUsingSamePartTwice}} (kk_0)^t.
		\end{align*}
		
		For every $i$ let us choose $r_i$ uniformly at random, and let $W'$ be the subgraph spanned by the chosen vertices. For $T\in \T^*$ let $A(T)$ be the event that all of the vertices of $T$ are in $W'$. Then
		%
		\begin{align*}
		\E[k_0^t \N(W',T)]&=\E\big[k_0^t \sum_{T\in \T^{*}} \mathbf{1}_{A(T)} \prod_{v_i,v_j\in E(T)} w^*(v_i,v_j)\big]= \\
		&=k_0^t \sum_{T\in \T^{*}}\frac{1}{k_0^t}\prod_{v_i,v_j\in E(T)} w^*(v_i,v_j)>\N(W^{*},T)-\epsilon_{\ref{lem:TCopiesUsingSamePartTwice}} (k\cdot k_0)^t
		\end{align*}
		Thus, there must be some choice of $r_1,...,r_k$ that gives the needed inequality.
	\end{proof}
	
	Given the graph $W^{**}$, that is the choice of $r_i$ for every $i$, we can now find an $\FF_{hom}$-free conventional subgraph of $W$, say $W_1$, for which $\N(W_1,T)>\N(W_0,T)$. 
	%
	%
	Take $(v_i,v_j)$ to be an edge of $W_1$ if and only if $(v_{i,r_i},v_{j,r_j})$ is an edge of $W^{**}$.
	First note that by (\ref{eq:isARegSet}) and the fact that $G_{\ex}$ is a subgraph of $G$ we get that
	\begin{align}\label{eq:diffInDens}
	w_1(v_i,v_j)&=d_G(V_i,V_j)
	\geq d_{G}(V_{i,r_i},V_{j,r_j})-\epsilon \nonumber \\
	&\geq d_{G_{\ex}}(V_{i,r_i},V_{j,r_j})-\epsilon=w^{**}(v_{i,r_i},v_{j,r_j})-\epsilon.
	\end{align}
	
	Second, the mapping $\varphi:V(W^{**})\to V(W_1)$, $\varphi(v_{i,r_i})=v_i$ maps edges with non-negative weight to edges with non-negative weight. Thus as $W^{**}$ is $\Hhom$-free then so is $W_1$. Finally, note that if $\T(W_1)$ and $\T(W^{**})$ are the copies of $T$ in $W_1$ and $W^{**}$ respectively, then $\varphi$ is a bijection between them. Thus, the following holds
	
	\begin{align*}
	\N(W_1,T)&=\sum_{T\in \mathcal{T}(W_1)} \prod_{(v_i,v_j)\in E(T)}w_{1}(v_i,v_j)\\
	&=\sum_{T\in \mathcal{T}(W^{**})} \prod_{(v_{i,r_i},v_{j,r_j})\in E(T)}w_{1}\big(\varphi^{-1}(v_{i,r_i}),\varphi^{-1}(v_{j,r_j})\big)\\
	&\overset{(\ref{eq:diffInDens})}{\geq}\sum_{T\in \mathcal{T}(W^{**})} \prod_{(v_{i,r_i},v_{j,r_j})\in E(T)}(w^{**}(v_{i,r_i},v_{j,r_j})-\epsilon)\\
	&\overset{\ref{lem:minusEpsilonAndMult}}{\geq} \N(W^{**},T)-\delta_{\ref{lem:minusEpsilonAndMult}}k^t\\ 
	&\overset{\ref{lem:choosingTheSubgraphForCont}}{>}(\frac{1}{k_0})^t \N(W^{*},T)-\epsilon_{\ref{lem:TCopiesUsingSamePartTwice}}k^t-\delta_{\ref{lem:minusEpsilonAndMult}}k^t\\
	&\overset{(\ref{eq:diffOfWeightedGraphs})}{\geq} \N(W_0,T)+(\delta_{\ref{lem:mainLemma}}-\delta_{\ref{lem:countingLemma}}-\epsilon_{\ref{lem:TCopiesUsingSamePartTwice}}-\delta_{\ref{lem:minusEpsilonAndMult}}) k^t
	>\N(W_0,T)
	\end{align*}
	The last inequality holds since $\delta_{\ref{lem:mainLemma}}$ can be chosen so that $\delta_{\ref{lem:mainLemma}}>\delta_{\ref{lem:countingLemma}}+\epsilon_{\ref{lem:TCopiesUsingSamePartTwice}}+\delta_{\ref{lem:minusEpsilonAndMult}}$.
	
	The existence of such $W_1$ is the needed contradiction, as we chose $W_0$ to be a conventional $\Hhom$-free subgraph of $W$ that has the maximum possible value of $\N(W_0,T)$. Thus there cannot be a graph $G_{\ex}$ which is $\FF$-free and has more than $ (\frac{n}{k})^{t}\ex_{hom}(W,T,\FF)+\delta_{\ref{lem:mainLemma}} n^{t}$ copies of $T$.
\end{proof}

\section{Proof of Theorem \ref{thm:algEx}}\label{sec:ProofOfTheorem14}


The algorithm itself is rather straightforward. Let $n_0=m_{\ref{lem:embedding}}K_{\ref{thm:fRegPart}}$, if $V(G)<n_0$ we can use brute-force as there is a constant number of options to check. If $V(G)>n_0$ then for $\epsilon_1 >0$ Theorem \ref{thm:fRegPart} gives us an efficient algorithm for finding $\epsilon_1 n^2$ edges to add or remove from $G$ and to find a $k\cdot K_{\ref{thm:Ref}}^{-1}$-regular partition $P=V_0\cup\{V_i\}_{i=1}^k$ of the edited graph, where $K_{\ref{thm:fRegPart}}> k>\frac{1}{\epsilon_1}$.

In Theorem \ref{thm:fRegPart} we change $\epsilon_{1} n^2$ edges, and as $|V_0|<K_{\ref{thm:fRegPart}}$ we can remove all edges using a vertex from $V_0$, and change at most $\epsilon_{1} n^2+K_{\ref{thm:fRegPart}}n<\epsilon_2 n^2$ edges of the original graph $G$. This will give us a graph that has an $f(k)$-partition of its vertices into $k$ parts of equal size. Call this graph $G^*$.
 
By lemma \ref{lem:removingSmallNumOfCopiesOfT} if we prove that we found the approximated answer for $G^*$, up to an additive error of $\epsilon_3n^t$, then we have also found an answer for $G$ itself up to an additive error term of $(\epsilon_2+\epsilon_3)n^t$. Thus we may focus on the graph $G^*$, remembering that
$$
|\ex(G,T,\FF)-\ex(G^*,T,\FF)|\leq (\epsilon_2+\epsilon_3)n^t
$$

Let $W$ be the $(\epsilon_1,d_{\ref{lem:embedding}})$ partition graph of $P$ and $G^*$. We use brute-force to find a conventional subgraph of it, say $W_0$, which is $\Hhom$-free and maximizes $\N(W_0,T)$. As $W$ has a constant number of vertices, we can do this in constant time (not depending on the size of $G$). By Lemma \ref{lem:F-Hom-free} $G^*_{W_0}$ is $\FF$-free, and so 
$$
\N(G^*_{W_0},T)\leq \ex(G^{*},T,\FF)
$$
To prove that $G^*_{W_0}$ gives the needed approximation, it is left to show that $$\N(G^*_{W_0},T)\geq \ex(G^{*},T,\FF)-\epsilon n^{t}.$$


Indeed Lemma \ref{lem:countingLemma} shows that $$|\N(G^*_{W_0},T)-(\frac{n}{k})^t\N(W_0,T)|<\delta_{\ref{lem:countingLemma}} n^{t}.$$ By Lemma \ref{lem:mainLemma} there is no $\FF$-free subgraph of $G^*$ that has more than $(\frac{n}{k})^t\N(W_0,T)+\delta_{\ref{lem:mainLemma}}n^t$ 
copies of $T$, thus $$|\ex(G^*,T,\FF)-(\frac{n}{k})^t\N(W_0,T)|<\delta_{\ref{lem:mainLemma}}n^t.$$



Thus we get that $|\N(G^*_{W_0},T)-\ex(G,T,\FF)|<(\epsilon_2+\epsilon_3+\delta_{\ref{lem:mainLemma}}+\delta_{\ref{lem:countingLemma}})n^t$, and as $\epsilon_2,\epsilon_3,\delta_{\ref{lem:mainLemma}}$ and $\delta_{\ref{lem:countingLemma}}$ all tend to zero as $\epsilon$ tends to zero, their sum can be made as small as needed. 
\qed

\section{Proof of Theorem \ref{thm:kmBigDiff} }\label{sec:ProofOfTheorem17}
	We show that the question of approximating $\ex(G,K_2,K_3)$ up to additive error $n^{2-\epsilon'}$ for any $\epsilon'>0$  can be reduced to the question of approximating $\ex(G,K_m,K_{k})$ up to additive error of $n^{m-\epsilon}$ for any $\epsilon>0$. As the former is NP-hard, as shown in \cite{ASS}, and this will show that so is the latter.
	
	To reduce one question to the other we do the following. Given a graph $G$ with vertex set $V$, add to it $r=k-3$ independent sets each of size $cn$, $U_1,...,U_{r}$, where $c$ is a large constant. Connect the vertices in a set $U_i$ to all vertices in sets $U_j$ with $i\ne j$ and to $V$. Call the new graph $G^+$, the new edges \textit{outer edges}, and the edges spanned by $V$ \textit{inner edges}.  
	
	Let $G^{+}_{ex}$ be a $K_k$-free subgraph of $G^{+}$ that has $\ex(G^+,K_m,K_k)$ copies of $K_m$. We show that the subgraph of $G^{+}_{ex}$ spanned by $V$ is  $K_3$-free, the set of edges removed from $G^{+}$ to get $G^+_{ex}$, say $E$, contains only inner edges  and it is of minimal possible size. In particular, this will show that $|E|=e(G)-\ex(G,K_2,K_3)$.
	
%
	Focusing on the graph $G^{+}_{ex}$ assume towards contradiction that there are copies of $K_{3}$ spanned by $V$.
	Note that if there is a copy of $K_{r}$ using vertices from $U_1,..,U_r$ then  $G^+_{ex}$ must miss at least one outer edge from the copy of $K_{r}$ to every copy of $K_3$ in $V$, as the graph is $K_k$-free.
	
	We first prove a lower bound on the number of copies of $K_m$ in $G^{+}$ using those missing outer edges, and then we show that if we restore all of the outer edges to $G_{ex}^+$ and make $V$ span a $K_3$-free graph then we would gain copies of $K_m$, in contradiction to the definition of $G_{ex}^{+}$. From this we can conclude that indeed $G_{ex}^+$ has all the outer edges, and that the subgraph of $G_{ex}^+$ spanned by $V$ is $K_3$-free.

	%
	
	\textbf{Removing outer edges, case 1:} Assume first that there is no set of $\frac{c}{2}n$ vertex disjoint copies of $K_{r}$ in $U_1,...,U_{r}$. If this is the case the number of edges between $U_1,...,U_{r}$ missing from $G_{ex}^+$ is at least $(\frac{cn}{2})^2$. 
	
	Indeed, take a maximal set of vertex disjoint copies of $K_{r}$, this set is of size at most $\frac{c}{2}n-1$. Let us look at all of the vertices not taking part in it. There are at least $\frac{c}{2}n+1$ such vertices in each $U_i$, call them $U'_i$, and let $e$ be the number of non-edges between them. 
	
	Choose randomly a single vertex from each $U'_i$. In the graph spanned by these vertices there must be at least one non-edge, because otherwise we would find another copy of $K_{r}$ in contradiction to the maximality of the set. Thus the expected number of non edges in the graph spanned by the vertices must be at least 1. This gives the following inequality and the needed result:
	\[
	\binom{r}{2}\frac{e}{\binom{r}{2}(\frac{cn}{2}+1)^2}\geq 1\Rightarrow e>\left(\frac{cn}{2}\right) ^2
	\]
	
	For each edge between the sets $U_i$ and $U_{i'}$ missing from $G_{ex}^+$, a choice of a single vertex from $m-3$ sets $U_j$ the edge does not take part in and a single vertex from $V$, spans together with this edge a copy of $K_m$ in $G^{+}$. There are $\binom{r-2}{m-3}(cn)^{m-3}n$ ways to make these choices, but going over all the missing edges it might be that we count each copy of $K_m$ several times, but certainly not more than $\binom{m-1}{2}$. Thus the number of copies of $K_m$ missing is at least:
	
	\begin{equation}\label{eq:removingOuter1}
	\binom{m-1}{2}^{-1}\left(\frac{cn}{2}\right) ^2\cdot\binom{r-2}{m-3} (cn)^{m-3}n= n^m c^{m-1} \left(\binom{r-2}{m-3}\binom{m-1}{2}^{-1}\frac{1}{4}\right)
	\end{equation}
	%
	%

	\textbf{Removing outer edges, case 2:} If there is a set of $\frac{c}{2}n$ vertex disjoint copies of $K_{r}$ in $U_1,...,U_r$, then in $G_{ex}^+$ we must miss at least one edge between each copy of $K_r$ and a vertex covering of $K_3$ in $V$. Let  $b$ be the size of a minimal vertex covering, then we miss at least $b\frac{cn}{2}$ edges. For each such edge if we choose a single vertex from $m-2$ sets $U_i$ it does not take part in, together they will span a copy of $K_m$. Thus, the number of copies of $K_m$ missing from $G_{ex}^{+}$, taking into account possible double counting, is at least
	\begin{equation}\label{eq:removingOuter2}
	(m-1)^{-1} b\frac{cn}{2}\binom{r-1}{m-2}(cn)^{m-2}=bn^{m-1}c^{m-1}\binom{r-1}{m-2}(2(m-1))^{-1}.
	\end{equation}
	
	Now let us compare the above to the option of removing from $G_{ex}^+$ all of the inner edges using vertices from a minimal vertex covering of the copies of $K_3$ in $V$ and restoring all of the outer edges.
	
	\textbf{Removing inner edges:}
	The number of inner edges removed from the graph is at most the edges connected to a minimal vertex covering of $K_3$ in $G$, that is, at most $bn$ edges. Even if for each such edge any choice of another $m-2$ vertices from the graph $G^{+}$ spans a copy of $K_m$, the number of copies of $K_m$ deleted when we remove these inner edges is at most:

	\begin{equation}\label{eq:removingInner}
	\binom{rcn+n}{m-2}bn<bn^{m-1}\left(rc+1\right)^{m-2}.
	\end{equation}
	
	
	Choosing the constant $c$ to be large enough and as $b<n$ we can ensure that (\ref{eq:removingInner})$<$(\ref{eq:removingOuter1}),(\ref{eq:removingOuter2}). Thus we get that indeed in $G_{ex}^{+}$ the subgraph spanned by $V$  is $K_{3}$-free. 
	
	
	It is left to show that the number of inner edges removed is  the minimum possible, thus solving $\ex(G,K_2,K_3)$. 
	Indeed, as the subgraph of $G_{ex}^+$ spanned by $V$ is $K_3$-free, all copies of $K_m$ using 3 or more vertices from $V$ are missing from $G_{ex}^+$. 
	As only inner edges are missing, all of the copies using 1 or less vertices from $V$ are kept in $G_{ex}^+$. As for copies of $K_m$ using exactly 2 vertices of $V$, for every edge missing from  $G_{ex}^+$ the number of copies of $K_m$ missing is exactly
	\[
	\binom{r}{m-2}(cn)^{m-2},
	\] 
	and copies corresponding to distinct edges are also distinct.
Thus, as the minimum possible number of copies of $K_m$ is missing from $G_{ex}^{+}$, it must be that the minimum possible number of edges is missing from $G$.

For integers $r,l$ and a graph $G$ let $\overline{\ex}(G,K_r,K_l)=\N(G,K_r)-\ex(G,K_r,K_l)$ where $\N(G,K_r)$ is the number of copies of $K_r$ in $G$. Note that the total number of copies of $K_r$ in $G$ can be found in polynomial time, and so if we can approximate $\overline{\ex}(G,K_r,K_l)$ up to an additive error in polynomial time, we can also approximate ${\ex}(G,K_r,K_l)$ in polynomial time up to the same additive error.

The above claim about $G^+$ can be restated as follows. Let $g$ be the number of copies of $K_m$ in $G^{+}$ using at least $3$ vertices from $V$, then 
\begin{align*}
\overline{\ex}(G^{+},K_m,K_k)&=\binom{r}{m-2}(cn)^{m-2}\overline{\ex}(G,K_2,K_3)+g\\
&\Downarrow\\
 \overline{\ex}(G,K_2,K_3)&=\frac{\overline{\ex}(G^{+},K_m,K_k)-g}{\binom{r}{m-2}(cn)^{m-2}}.
\end{align*}

As $g$ and $|E(G)|$ can be calculated in polynomial time, if we approximate $\overline{\ex}(G^{+},K_m,K_k)$ up to an additive error of $ n^{m-\epsilon}$ then we can also approximate $\ex(G,K_2,K_3)=|E(G)|- \overline{\ex}(G,K_2,K_3)$ up to an additive error of 
$$
n^{m-\epsilon}\Big(\binom{r}{m-2}(cn)^{m-2}\Big)^{-1}\leq n^{2-\epsilon'}
$$
for some $\epsilon':=\epsilon'(\epsilon)>0$ and this is known to be NP-hard. 
\qed
%
%
%
%

\section{Concluding remarks and open questions}\label{sec:concludingRemarks}

\subsection{Improving on Theorem \ref{thm:kmBigDiff}}

In the introduction we have posed the following conjecture,

\noindent \textbf{Conjecture \ref{conj:mainConj}}	\textit{
	For every graph $T$, family of graphs $\FF$ such that no $F\in \FF$ is a subgraph of a blow-up of $T$, and $\epsilon>0$, it is NP-hard to approximate $\ex(G,T,\FF)$ up to additive error  of $n^{v(T)-\epsilon}$ for a given input graph $G$ on $n$ vertices.}

Allowing a  smaller additive error, we can get the above for a large family of graphs.

\begin{proposition}\label{prop:toSmallofAnError}
	For any $\epsilon>0$, 3-connected graph $T$ on $t$ vertices and a family of 3-connected graphs $\FF$ such that no $F\in \FF$ is a subgraph of a blow-up of  $T$, it is NP-hard to approximate $\ex(G,T,\FF)$ up to additive error of $n^{t-2-\epsilon}$. 
\end{proposition}

\begin{proof}
	We show that the question of calculating $\ex(G,K_2,\FF)$ exactly can be reduced to the question of approximating $\ex(G,T,\FF)$ up to additive error of $n^{c(t-2)/(c+2)-\epsilon'}$ where $c$ is a constant as large as needed and $0<\epsilon'<\epsilon$. As every $F\in \FF$ is not a subgraph of a blow-up of $T$, in particular it is not bipartite, and thus the former is known to be NP-hard by Theorem \ref{thm:ASS2}. As $c$ can be as large as we want it can be chosen so that $c(t-2)/(c+2)-\epsilon'>t-2-\epsilon$.

	Given a graph $G$ on $n$ vertices, let us construct $G^{+}$ as follows. Let $T'$ be a subgraph of $T$ obtained by removing two arbitrary vertices connected by an edge, say $\{u,v\}$. For each edge $e\in G$ add $t-2$ independent sets of size $n^{c}$ to the graph. Connect the appropriate pairs of sets by complete bipartite graphs to create a blow-up of $T'$, and then connect   sets corresponding to neighbors of $u$ and $v$ in $T$ to the two endpoints of $e$.   
	Call the new copies of $T$ created \textit{external} and the copies spanned by $G$ \textit{internal}.
	
	%
	Note that $N:=v(G^{+})\leq n+(t-2)n^{c}n^2$ and $N^{c(t-2)/(c+2)-\epsilon'}<\frac{1}{5}n^{c(t-2)}$. 
	Furthermore, in $G^{+}$ every edge of $G$ takes part in a fixed number of  external copies of $T$ in $G^+$, say $ X_T\geq n^{c(t-2)}$ and in at most $O(n^{t-2})$ internal copies. In addition, no new copies of  graphs from $\FF$ are created. This is true as $T$ and the graphs in $\FF$ are 3-connected and no graph in $\FF$ is a subgraph of a blow-up of $T$.
	
	Let $\overline{\ex}(G^{+},T,\FF)=\N(G^+,T)-\ex(G^+,T,\FF)$ where $\N(G^+,T)$ is the number of copies of $T$ in $G^+$, and similarly $\overline{\ex}(G,K_2,\FF)=|E(G)|-\ex(G,K_2,\FF)$. First note that
	\begin{align*}
		\overline{\ex}(G^{+},T,\FF)&\leq \overline{\ex}(G,K_2,\FF)X_T+O(n^{t}).
	\end{align*} 
	Indeed, by deleting $\overline{\ex}(G,K_2,\FF)$ edges from $G$ we can make $G^{+}$ into an $\FF$-free graph, as all of the copies of $\FF$ are spanned by the vertices coming from $G$. Removing these edges will remove all of the external copies of $T$ using them together with some internal copies. 
	
	Furthermore, if we removed $\overline{\ex}(G^{+},T,\FF)$ copies of $T$ from $G^{+}$ and made it $\FF$-free, we may assume that we have done this by removing only edges from $G$, say $e$ of them, and that $G$ was made $\FF$-free. Each edge of $G$ takes part in at least $X_T$ distinct copies of $T$, so
	$$
	\overline{\ex}(G^{+},T,\FF)\geq e X_T
	\geq \overline{\ex}(G,K_2,\FF)X_T.
	$$
	
	From the above
	
	\[
	\frac{\overline{\ex}(G^{+},T,\FF)-O(n^t)}{X_T}\leq \overline{\ex}(G,K_2,\FF) \leq \frac{\overline{\ex}(G^{+},T,\FF)}{X_T},
	\]	and as $$\frac{N^{c(t-2)/(c+2)-\epsilon'}}{X_T}\leq \frac{1}{5} \mbox{ and } \frac{O(n^t)}{X_t}<\frac{1}{5},$$
	if we calculated $\overline{\ex}(G^{+},T,\FF)$ up to an additive error of $N^{c(t-2)/(c+2)-\epsilon'}$ then we calculated $\overline{\ex}(G,K_2,\FF)$ up to an additive error of $2/5$, as this is an integer this means we have calculated it exactly and this is known to be NP-hard.
\end{proof}

The full assertion of Conjecture \ref{conj:mainConj} remains open. The following questions address several special cases.

\begin{enumerate}
\item Is approximating $\ex(G,K_m,K_{m+1})$ up to additive error of $n^{m-\epsilon}$ NP-hard for any integer $ m\geq 2$ and any $\epsilon>0$? The case $m=2$ is proved in \cite{ASS} and we can also prove it for $m=3$.
\item Given a family of graphs $\FF$ such that  $\chi(F)\geq m+2$ for every $F\in \FF$,
is it NP-hard to approximate $\ex(G,K_m,\FF)$ up to additive error of $n^{m-\epsilon}$ for every $\epsilon>0$? Theorem \ref{thm:kmBigDiff} is of course a special case of this.
 
\end{enumerate}

%
%
%
%

\subsection{Calculating $\ex(G,T,F)$ exactly}
Proposition \ref{prop:toSmallofAnError} implies that for many graphs $T$ and $F$ there is no efficient algorithm that calculates $\ex(G,T,F)$ exactly. Nevertheless, for some special cases this calculation is possible in polynomial time.
We mention two simple examples.

\begin{proposition}
	For a graph $G$ on $n$ vertices the following can be solved in polynomial time
	\begin{enumerate}
		\item $\ex(G,kK_2,K_{1,2})$, where $kK_2$ is a matching of size $k\geq1$.
		\item $\ex(G,K_2,K_{1,{t+1}})$.
	\end{enumerate}
\end{proposition}

\begin{proof}

		Part 1 is trivial. It is known since Edmonds \cite{E} that given an input graph $G$ a matching of maximum size can be found in polynomial time. As any $K_{1,2}$-free subgraph of $G$ is a matching it is clear that maximizing the largest matching also maximizes the number of copies of $kK_2$.
		
		The proof of part 2 follows the idea of the proof of the f-factor theorem of Tutte \cite{T}. Given a graph $G$, we may assume that it has no isolated vertices. First, replace each vertex $v\in V(G)$ with an independent set of size $d(v)$, say $V(v)=\{v_1,...,v_{d(v)}\}$, and for every edge $e=\{u,v\}$ of $G$ choose arbitrarily vertices $v_i,u_j$ and connect them, making sure that at the end of the processes each vertex of the new graph is of degree exactly one. Note that there is a one-to-one correspondence between edges in the new graph and edges in $G$.
		
		Second, for every independent set $V(v)$ corresponding to a vertex $v\in V(G)$ such that $d(v)\geq t+1$, add a new independent set $U(v)$ of  size $d(v)-t$. Connect all of the vertices in $V(v)$ to all the vertices in $U(v)$.  Call the new graph obtained by this processes $G^*$.
		
		By \cite{E} there is a polynomial time algorithm that finds a maximum matching in $G^*$, call this maximum matching $M$. Note that we may assume that $M$  saturates all  the vertices in the sets $U(v)$. Indeed, if some $w\in U(v)$ is not saturated we may add to $M$ an edge between $w$ and some vertex of $V(v)$, say $v_1$, that is not adjacent to $U(v)$ in $M$ and if prior to this addition there was an edge in $M$ adjacent to $v_1$, delete it. A vertex $v_1$ must exist as $V(v)>U(v)$ and the replacement of edges does not make the size of $M$ smaller. 
		
		Each edge in $M$ that is not adjacent to some $U(v)$ corresponds to an edge of $G$. Let us keep in $G$ only these edges and call the new graph $G'$. For every vertex $v$ the set $U(v)$ was saturated in $M$ and so the number of edges adjacent to $V(v)$ and a vertex not in $U(v)$ is at most $|V(v)|-|U(v)|=d_G(v)-(d_G(v)-t)=t$. Thus the new graph $G'$ is a subgraph of $G$ with maximum degree at most $t$. The number of edges in $G'$ is exactly the number of edges in $M$ that are not adjacent to some $U(v)$, that is $e(G')=|M|-\sum_{v\in V(G)}(d_G(v)-t)$.
		
		It is left to show that $G'$ has the maximum possible number of edges. Indeed, assume that $G^{''}$ is a subgraph of $G$ with maximum degree at most $t$ which has more edges than $G'$. Looking at $G^*$ take into a matching $M'$ each edge corresponding to an edge in $G''$. This results in a matching as each $v_i\in V(v)$ has exactly one neighbor outside of $U(v)$. In addition, for each $v\in V(G)$ choose $d_G(v)-t$ edges between $U(v)$ and $V(v)$ to add to $M'$ while keeping it a matching. As the maximum degree in $G''$ is at most $t$ there will be $|U(v)|=d_G(v)-t$ unsaturated vertices in $V(v)$ to connect to $U(v)$. The resulting matching will be of size $e(G'')+\sum_{v\in V}(d_{G}(v)-t)>|M|$,  in contradiction to the maximality of $M$.
		
		Thus we get that $\ex(G,K_2,K_{1,t+1})=e(G')=|M|-\sum_{v\in V(G)}(d_G(v)-t)=|M|-\frac{1}{2}e(G)+v(G)t$, and we can find $|M|, e(G)$, and $v(G)$ in polynomial time.
		

\end{proof}

It will be interesting to characterize all pairs of graphs $T$ and $F$ for which $\ex(G,T,F)$ can be calculated exactly in polynomial time, for a given input graph $G$.

\end{document}